\documentclass[preprint,1p, 11pt, times]{elsarticle}

\usepackage{amssymb}
\usepackage{amsmath}

\journal{Chaos, Solitons \& Fractals}

\usepackage{amsthm,booktabs}

\newtheorem{theorem}{Theorem}
\newtheorem{lemma}{Lemma}
\newtheorem{corollary}{Corollary}
\newtheorem{definition}{Definition}
\newtheorem{remark}{Remark}
\newtheorem{proposition}{Proposition}

\newcommand{\Mellin}{\mathcal{M}}
\newcommand{\Res}{\operatorname{Res}}
\newcommand{\Poles}{\operatorname{Poles}}

\usepackage{lineno}

\begin{document}

\begin{frontmatter}



\title{Mellin-Space Prony  Representability of Linear Viscoelastic Models}


\author{Dimiter Prodanov}

\affiliation{organization={Laboratory of Neurotechnology, PAML-LN, Institute of Information and Communication Technologies, Bulgarian Academy of Science},
            addressline={2, acad. G.Bontchev Str, bl. 25A}, 
            city={Sofia},
            postcode={1113}, 
            state={},
            country={Bulgaria}}


\begin{abstract}
Linear viscoelastic materials are commonly described by continuous relaxation spectra, 
yet practical measurements and simulations employ discrete Prony series. 
In the Laplace frequency domain, the distinction is well 
understood: rational transfer functions admit finite Prony representations, while 
fractional models with branch cuts do not. 
This work provides a complementary and structurally deeper characterization in the 
Mellin transform domain. We prove that a viscoelastic modulus admits an exact finite 
Prony series if and only if the arithmetic pole lattices of its Mellin kernel align 
with the integer lattice of the constitutive kernel, and the associated residues 
satisfy decoupled first-order recurrences along aligned sublattices. 
Unlike the Laplace-domain rational/non-rational dichotomy, the Mellin criterion reveals the 
arithmetic geometry underlying finite representability, which requires Diophantine alignment of infinite pole progressions and the compatibility of their residues.
Applying this criterion yields a complete model taxonomy. 
Classical spring-dashpot networks 
(Maxwell, standard linear solid) satisfy the alignment and recurrence conditions. 
In contrast, fractional models (power-law, Cole-Cole, Havriliak-Negami, Zener) and log-normal spectra 
violate one or both conditions and require infinite Prony ladders for exact representation. 
The framework thus shifts the question of finite network realizability from an algebraic condition on rational functions to a geometric condition on pole lattices, offering both a theoretical classification and a 
practical computational test.
\end{abstract}



\end{frontmatter}



\section{Introduction}
A material is called linear viscoelastic if its stress is a linear time-dependent functional of material strain.  
Linear viscoelasticity is described by relaxation kernels, which are continuous  distributions of relaxation times. Yet experimentally, measurements are necessarily discrete and band-limited, creating a fundamental gap between the continuous mathematical description and empirical observations.
Prony series provide a widely used approximate representation of any viscoelastic model through finite sums of exponential terms in its time-dependent relaxation shear and compression moduli.
A comprehensive recent review of Prony-series applications in viscoelasticity is given in~\cite{Hristov2019}. 
With the caveat that ``all models are wrong but some are useful'' \cite{Box1976} --  a stance widely adopted in statistics and data science -- this work explores the utility of Prony series in describing linear viscoelasticity from a complex-analytic perspective.

From a dynamical-systems perspective, finding a linear viscoelastic representation is equivalent to constructing a mechanical network--a finite arrangement of springs and dashpots--whose input-output behavior reproduces the prescribed complex-valued dynamic modulus $G^*(\omega)$. Such networks generate relaxation spectra that are finite sums of Dirac masses (finite Prony series) with rational transfer functions in the frequency (Fourier/Laplace) variable. 
The inverse question -- whether a given modulus admits a finite network representation -- translates into a classical realization problem: does the system admit a finite-dimensional state-space description, or equivalently, does its transfer function have only finitely many simple poles on the negative real axis? The Mellin-transform framework developed here recasts this condition as an algebraic constraint on the pole lattices of the transformed kernel, thereby linking complex-plane singularity geometry directly to the existence of finite spring-dashpot networks.

Classical spring-dashpot networks -- Maxwell, Kelvin-Voigt, Standard Linear Solid, and their finite generalizations -- yield linear ordinary differential  equations for the constitutive relations. 
Furthermore, their complex moduli $G^*(\omega)$ are rational functions of the  spectral variable in the Fourier/Laplace domains, admitting exact  Prony representations with a finite number of exponential memory terms. 
That is, in the frequency domain their spectrum contains a finite number of poles.  

Fractional constitutive laws, by contrast, are both memory-dependent and non-local: stress at a given instant depends on the entire prior strain history, not merely on instantaneous rates. This is formalized through differintegral operators \cite{Oldham1974} -- generalized combinations of differentiation and integration to non-integer order -- which yield constitutive relations that cannot be reduced to ordinary differential equations \cite{Mainardi2010}.
Many soft materials exhibit power-law viscoelastic responses across typical observation bandwidths, modeled by continuous distributions of relaxation times that reflect their fractal internal structure  \cite{Bonfanti2020}.
Structurally, power-law models, Cole-Cole, and Havriliak-Negami kernels generate non-rational dynamic moduli $G^*(\omega)$ with branch-cuts, implying infinite relaxation times distributions.

Recent works survey Prony-series approximations of fractional-model data \cite{Hristov2019, Hristov2022}. However, these approximations remain \emph{ad hoc}: no systematic procedure exists to derive the Prony coefficients directly from the constitutive structure of a given model.
To understand the dichotomy between finite Prony models and non-local fractional models one needs to address the inverse problem:
given an analytical expression of $G^*(\omega)$, does it admit a faithful finite Prony representation? This question is numerically ill-posed: recovering $\{g_k, \tau_k\}$ from frequency data suffers Hadamard instability, where noise amplifies into spurious oscillations, obscuring the discrete-vs-continuous distinction \cite{Thigpen1983,Stankiewicz2023}.
Therefore, to establish the answer one needs to employ analytical tools. 

We address this dichotomy through Mellin-space analysis. 
In the Mellin complex plane fractional models have kernels (i.e. Mellin symbols) containing non-degenerate Gamma factors, whose poles form infinite arithmetic progressions. 
We prove that  exact finite Prony representability holds \emph{if and only if} these Gamma-factor pole lattices align exactly with the poles of the trial-state spectrum, with residues satisfying decoupled first-order recurrence relations along sublattices.
This maximal criterion classifies classical models (Maxwell, SLS) as finitely representable and fractional models (power-law, Cole-Cole, Gaussian) as transcendentally representable via infinite Prony ladders.

The paper is organized as follows.  
General notation and conventions are stated in Section~\ref{sec:notation-and-conventions}.  
Section~\ref{sec:mellin} introduces the Mellin-space formulation of linear viscoelasticity.  
Section~\ref{sec:criteria} states and proves the main criterion-a necessary and sufficient condition for finite representability within the Extended Fox class \(\mathcal{Q}\).  
Section~\ref{sec:classif} classifies common viscoelastic models according to this criterion.  
Section~\ref{sec:test} provides a practical computational test for finite representability.  
Technical details are relegated to the Appendix. 

\section{Notation and Conventions}\label{sec:notation-and-conventions}

\subsection{Mellin transforms}

For a function \( f(t) \) integrable on the positive half-line \((0,\infty)\) whose growth near 0 and about infinity is at most power-law, the Mellin transform is defined by the complex-valued integral
\[
\Mellin\{f\}(s) \equiv \tilde f(s) := \int_0^\infty t^{\,s-1} f(t)\,dt,
\]
whenever the integral converges absolutely in some vertical strip \( a < \Re s < b \) of the complex plane \cite{Oberhettinger1974}. For many functions arising in viscoelasticity (e.g., exponentially decaying relaxation spectra), such a strip exists.

\textbf{Extension to generalized functions.}  
The definition extends to distributions via analytic continuation, and to tempered distributions via the theory of generalized functions.
In particular:
\begin{itemize}
	\item For a \emph{discrete spectrum} coming from finite Prony series \( H(\tau) = \sum_{k=1}^{N} g_k \tau_k \delta(\tau - \tau_k) \), the Mellin transform is evaluated as the finite sum
	\[
	\tilde H(s) = \sum_{k=1}^{N} g_k \tau_k^{\,s},
	\]
	which is an entire function of exponential type.
	
	\item For a \emph{pure power-law} \( f(\tau) = \tau^{\beta} \), the Mellin transform 
	is defined via analytic continuation from the strip $(0, 1)$ to the strip $[1, \infty) $ of the range-limited integral 
	\[
	\Mellin \{\tau^\beta\} (s)= \int_{0}^{1} \tau^{\beta} \tau^{s-1} d \tau =\frac{1}{\beta +s}
	\]
	and corresponds to a simple pole at \(s = -\beta\).
	
\end{itemize}

All Mellin transforms appearing in this work are either classical (convergent in a strip) or defined via such analytic/distributional extensions, ensuring that the subsequent constitutive identity and pole-lattice analysis are mathematically well-defined.

\subsection{The Discrete Prony Class \(\mathcal{P}\) and the H-Function Hierarchy}\label{sec:prony}

The canonical Prony model \cite{Prony1795},
\[
G_P(t)=G_\infty + \sum_{k=1}^N g_k e^{-t/\tau_k},
\]
represents the stress relaxation as a finite superposition of decreasing exponentials, reflecting the fading-memory principle of Boltzmann's superposition~\cite{Baumgaertel1992}. 
Furthermore, Bernstein's theorem~\cite{Bernstein1929} guarantees that completely monotone relaxation moduli $G(t)$ admit continuous spectral representations as non-negative Laplace transforms:
\begin{equation}\label{eq:Gtime}
	G(t) = G_\infty + \int_0^\infty e^{-t/\tau} H(\tau) \frac{d\tau}{\tau},\quad H(\tau)\geq 0. 
\end{equation}
This establishes realizability for fading-memory materials but leaves unresolved the structural dichotomy between \emph{finite} Prony series (discrete spectra) and generic continuous distributions.

To address this dichotomy analytically, we introduce a hierarchy of three function classes based on the structure of the Mellin transform $\widetilde{H}(s) = \Mellin\{H(\tau)\}(s)$.

\begin{definition}[Discrete Prony class \(\mathcal{P}\)]\label{def:prony}
	A function \(\tilde H(s)\) belongs to the \emph{Discrete Prony class} \(\mathcal{P}\) if it is a finite exponential sum:
	\[
	\tilde H(s) = \sum_{k=1}^N g_k e^{\lambda_k s}, \qquad g_k > 0,\; \lambda_k \in \mathbb{R}.
	\]
	Equivalently, its inverse Mellin transform is a finite sum of Dirac masses $$H(\tau) = \sum_{k=1}^N g_k \tau_k \delta(\tau-\tau_k)$$ with $\tau_k = e^{\lambda_k}$.
\end{definition}

Many viscoelastic models, particularly those arising from fractional constitutive laws, possess \emph{continuous} relaxation spectra whose Mellin transforms involve ratios of Gamma functions. These are precisely the Mellin transforms of Fox's H-functions \cite{MathaiSaxenaHaubold2010, Fox1961}.

\begin{definition}[H-function class \(\mathcal{H}\)]\label{def:Hclass}
	A function \(\tilde H(s)\) belongs to the \emph{H-function class} \(\mathcal{H}\) if it can be written as a finite ratio of Gamma functions:
	\[
	\tilde H(s) = \frac{\prod_{i=1}^{N}\Gamma(a_i s+b_i)}{\prod_{j=1}^{M}\Gamma(c_j s+d_j)},
	\qquad a_i,c_j>0;\quad N,M<\infty.
	\]
\end{definition}

For the most general analysis within our framework, we allow multiplication by the exponential of an entire function. This extension is \emph{narrower} than the \(\overline{H}\)-function of Inayat-Hussain \cite{InayatHussain1987}, which admits \(s^s\) factors and encompasses polylogarithms and the Lambert \(W\) function. Our restricted class excludes such transcendental elements and is characterized solely by the factor \(Q(s) = \exp(P(s))\) where \(P(s)\) is entire.

\begin{definition}[Extended Fox class \(\mathcal{Q}\)]\label{def:Qclass}
	A function \(\tilde H(s)\) belongs to the \emph{Extended Fox class} \(\mathcal{Q}\) if it can be written as
	\[
	\tilde H(s) = H_\Gamma(s) \cdot Q(s),
	\]
	where \(H_\Gamma(s) \in \mathcal{H}\) is a finite Gamma-ratio and 
	\(Q(s) = \exp(P(s))\) with \(P(s)\) an entire function.
\end{definition}

The three classes form a strict hierarchy:
\[
\mathcal{P} \subset \mathcal{Q}, \qquad \mathcal{H} \subset \mathcal{Q}.
\]
Functions in \(\mathcal{P}\) are entire (pole-free), whereas non-degenerate members of \(\mathcal{H}\) possess infinite lattices of poles. The class \(\mathcal{Q}\) is the maximal class for which the pole-lattice analysis of this paper applies.

Some remarks are also in order about the \textbf{pole set} of \(\tilde H (s)\), designated as \(\Poles (\tilde H)\).
For any \(\widetilde{H} \in \mathcal{Q}\), this set is entirely determined by the H-function kernel \(H_\Gamma(s)\): each numerator Gamma factor \(\Gamma(a_i s+b_i)\) generates an infinite arithmetic progression $\Lambda_i$ with members
\[
\Lambda_i=\Bigl\{-\frac{b_i+k}{a_i}\;:\;k=0,1,2,\dots\Bigr\},\qquad \Delta_i=\frac{1}{a_i},
\]
which we call informally a \emph{pole lattice}. 
After cancellations between numerator and denominator lattices, the resulting set \(\Poles (\tilde H)\) remains a finite union of such lattices. 
Furthermore, since \(\Poles (\tilde H)\) is countable and totally disconnected, it is also nowhere dense on the real line.

\section{Mellin-Space Formulation of Linear Viscoelasticity}\label{sec:mellin}

We adopt the following notation conventions: $t$ denotes physical time, $\tau>0$ denotes relaxation timescale, $\omega>0$ is angular frequency, and $s$ is the complex Mellin variable, coupling the relaxation timescale to frequency.
The Mellin variable $s$ reveals the hidden logarithmic structure of relaxation spectra, transforming scale-separated Prony series into algebraic pole-lattice geometry.

\subsection{Constitutive relation}

The non-aging, isothermal stress-strain relation for a linear viscoelastic material is given by the Boltzmann superposition integral:
\begin{equation}
	\sigma(t) = \int_{-\infty}^t G(t-u) \, d\varepsilon(u),
\end{equation}
where $\sigma(t)$ and $\varepsilon(t)$ are the stress and strain, and $G(t)$ is the relaxation modulus.

In the time domain, the relaxation modulus of a linear viscoelastic solid takes the spectral form \eqref{eq:Gtime} as discussed. 
In the frequency domain, the relation becomes
\begin{equation}
	\sigma^*(\omega) = G^*(\omega) \varepsilon^*(\omega),
\end{equation}
where the complex modulus admits the spectral representation
\begin{equation}
	G^*(\omega) = G_\infty + \int_0^\infty \frac{i\omega\tau}{1 + i\omega\tau} H(\tau) \frac{d\tau}{\tau}.
	\label{eq:constitutive}
\end{equation}
The spectral kernel is thus identified as
\begin{equation}
	K(\omega\tau) := \frac{i\omega\tau}{1 + i\omega\tau}.
\end{equation}

As a technical condition we will require that the integral converges absolutely for every \(\omega>0\) and that both \(G^*\) and \(H\) possess Mellin transforms in overlapping vertical strips of the complex plane.

Observe that a \emph{finite Prony series} corresponds to a discrete spectrum
\begin{equation}
	H(\tau)=\sum_{k=1}^{N}g_k\tau_k\delta(\tau-\tau_k), 
\end{equation}
producing 
\begin{equation}
	G(t)=G_\infty+\sum_{k=1}^{N}g_k e^{-t/\tau_k}
\end{equation}
in the time domain.

\subsection{Mellin transform of the spectral kernel}

The viscoelastic spectral  kernel $K(\omega \tau)$ has Mellin transform (w.r.t. $\omega$, fixed $\tau>0$)
\begin{equation}\label{eq:kernel_mellin_section}
	\tilde K(s) = \Mellin_\omega\{K(\omega \tau)\}(s) = \pi e^{i\pi(1-s)/2} \csc(\pi s),
	\qquad -1<\Re s<0,
\end{equation}
as shown in Appendix Lemma~\ref{lem:beta}.

\begin{remark}
	The factor $e^{i\pi(1-s)/2}$ is entire (no poles). All singularities come from 
	$\csc(\pi s)$, which has simple poles at $s=n\in\mathbb{Z}$ with residues 
	$(-1)^n$. Thus $\tilde K(s)$ has simple poles at all integers $s=n$ with residues
	$-e^{i\pi(1-n)/2}(-1)^n$.
\end{remark}

\subsection{Constitutive equation in Mellin space}
To avoid dealing explicitly with the Dirac $\delta$ distribution let $\Delta G^* (\omega)= G^*(\omega)-G_\infty $, be the viscoelastic part of the complex modulus, and
\begin{equation}
	\tilde G(s) := \Mellin\{ \Delta G^*(\omega) \}(s)
\end{equation}
be its Mellin transform, and
$
\tilde H(s) = \Mellin\{H(\tau)\}(s) 
$
be the Mellin transform of the relaxation spectrum.
Applying the Mellin transform to \eqref{eq:constitutive} and using the convolution theorem for Mellin transforms yields 
\begin{equation}\label{eq:mellin_constitutive}
	\tilde G(s) = \tilde K(s) \,\tilde H(-s),
	\qquad 0 < \Re s < 1.
\end{equation}
Equation \eqref{eq:mellin_constitutive} is the fundamental Mellin-space constitutive relation of linear viscoelasticity. 
It links the unknown spectrum \(\tilde H(s)\) to the prescribed modulus \(\tilde G(s)\) through a simple multiplicative kernel that is meromorphic in \(s\).
Note that since $\tilde{K}(s)$ is odd, its argument sign convention corresponds to $\tilde{H}(-s)$.
This sign convention preserves the physical duality between the timescale variable $\tau$ and the frequency $\omega$ in the Mellin representation.

Note also the argument inversion in the Mellin representation \(\tilde H(-s)\), which determines that  $\tilde H$ has poles in the set
$
\Lambda_i^- = \{\lambda : -\lambda \in \Lambda_i\} 
$
i.e. the lattices are reflected through the origin.

\section{Criteria for Finite Prony Representation}\label{sec:criteria}

First, we consider the particular case of a degenerate $H(\tau)$ kernel.
\subsection{Degenerate case}
\begin{proposition}[Characterization of finite Prony spectra]\label{prop:finitePronyChar}
	Let \(H(\tau)\) be a positive measure on \((0,\infty)\) with Mellin transform \(\tilde{H}(s)\) converging in some vertical strip. Then \(H(\tau)\) is a finite Prony series, i.e.,
	\[
	H(\tau)=\sum_{k=1}^{N} g_k\tau_k\delta(\tau-\tau_k), \quad g_k>0,\ \tau_k>0,
	\]
	if and only if \(\tilde{H}(s)\) is a finite exponential sum of the form
	\[
	\tilde{H}(s)=\sum_{k=1}^{N} g_k e^{\lambda_k s},
	\]
	where  \(\lambda_k = \ln\tau_k \in \mathbb{R}\) .
\end{proposition}
\begin{proof}
	($\Rightarrow$) Direct computation gives \(\tilde{H}(s)=\sum_{k} g_k \tau_k^{s} = \sum_{k} g_k e^{s\ln\tau_k}\).
	
	($\Leftarrow$) Suppose \(\tilde{H}(s)=\sum_{k=1}^{N} g_k e^{\lambda_k s}\) with distinct \(\lambda_k\in\mathbb{R}\) and \(g_k >0 \). The inverse Mellin transform along a vertical line \(\Re s = \sigma\) in the strip of convergence yields
	\[
	H(\tau)=\frac{1}{2\pi i}\int_{\sigma-i\infty}^{\sigma+i\infty} \tau^{-s} \tilde{H}(s)\,ds
	= \sum_{k=1}^{N} g_k \cdot \frac{1}{2\pi i}\int_{\sigma-i\infty}^{\sigma+i\infty} \tau^{-s} e^{\lambda_k s}\,ds.
	\]
	The inner integral is the inverse Mellin transform of \(e^{\lambda_k s}\). Since \(e^{\lambda_k s} = \mathcal{M}\{\delta(\ln\tau-\lambda_k)\}(s)\) in the distributional sense, we obtain formally
	\[
	\frac{1}{2\pi i}\int_{\sigma-i\infty}^{\sigma+i\infty} \tau^{-s} e^{\lambda_k s}\,ds = \delta(\ln\tau - \lambda_k) = \tau\,\delta(\tau-\tau_k),
	\]
	using the scaling property of the delta function (valid in distributional sense) 
	\[
	\delta(g(\tau)) = \sum_{\tau_i: \, g(\tau_i)=0} \frac{\delta (\tau-\tau_i)}{|g'( \tau_i) |}
	\]
	where \(\tau_k = e^{\lambda_k}\). Hence \(H(\tau)=\sum_{k} g_k\tau_k\delta(\tau-\tau_k)\), which is a finite Prony series.
\end{proof}

\begin{proposition}[Finite Prony solutions in \(\mathcal{Q}\)]\label{cor:finiteProny}
	Assume the following \emph{genericity conditions}:
	\begin{enumerate}
		\item All poles of \(\widetilde{G}(s)\) are simple.
		\item The only pole coincidences are those forced by the lattice-alignment condition.
	\end{enumerate}
	Then a solution \(\widetilde{H}(s) \in \mathcal{Q}\) 
	corresponds to a finite Prony series if and only if:
	\begin{enumerate}
		\item The Gamma-ratio \(H_\Gamma(s)\) is a constant, and
		\item The entire factor \(Q(s)\) is a finite exponential sum \(\sum_{k=1}^N g_k e^{\lambda_k s}\) with \(\lambda_k \in \mathbb{R}\).
	\end{enumerate}
\end{proposition}

\begin{proof}
	By Proposition~\ref{prop:finitePronyChar}, a relaxation spectrum \(H(\tau)\) is a finite 
	Prony series if and only if its Mellin transform \(\widetilde{H}(s)\) is a finite exponential sum.
	
	\paragraph{Sufficiency.} If \(H_\Gamma(s)\) is constant (without loss of generality, 
	\(H_\Gamma(s) \equiv 1\)) and \(Q(s) = \sum_{k=1}^N g_k e^{\lambda_k s}\), then 
	\(\widetilde{H}(s) = Q(s)\) is a finite exponential sum. By Proposition~\ref{prop:finitePronyChar}, 
	the spectrum is a finite Prony series.
	
	\paragraph{Necessity.} Suppose \(\widetilde{H}(s)\) corresponds to a finite Prony series. 
	Then \(\widetilde{H}(s)\) is a finite exponential sum, hence an entire function with no poles. 
	Since \(\widetilde{H}(s) = H_\Gamma(s) Q(s)\) and \(Q(s) = \exp(P(s))\) never vanishes, 
	every pole of \(H_\Gamma(s)\) would be a pole of \(\widetilde{H}(s)\). Because \(\widetilde{H}(s)\) 
	has no poles, \(H_\Gamma(s)\) must have no poles. A Gamma-ratio without poles is necessarily 
	constant. With \(H_\Gamma(s) \equiv 1\), we have \(\widetilde{H}(s) = Q(s)\), so \(Q(s)\) 
	must be a finite exponential sum.
\end{proof}

\subsection{Generic case}

To proceed with the analysis we construct a trial-state ansatz defined as follows:
\begin{definition}[Trial  state]\label{def:trialstate}
	A trial  state  $G^*  $ is a modulus which admits a Mellin symbol with factorization:
	\begin{equation}\label{eq:trialstate}
		\Mellin\{G^*(\omega)\}(s)= \tilde G(s):=A(s)\,\Gamma(\alpha s+\beta),\qquad \alpha>0,
	\end{equation}
	where \(A(s)\) is either entire or meromorphic with a pole lattice disjoint from the lattice of $\Gamma(\alpha s+\beta)$.	
\end{definition}
Observe that a trial-state is finite Prony by construction and the Gamma factor of the trial-state produces a reference pole lattice
\[
\Lambda_G=\Bigl\{-\frac{\beta+m}{\alpha}\;:\;m=0,1,2,\dots\Bigr\},\qquad \Delta_G=\frac{1}{|\alpha|}.
\]

The argument of the main Theorem can be illustrated in Fig. \ref{fig:ladder_thm}.
\begin{figure}[h!]
	\centering
	\includegraphics[width=1.0\linewidth]{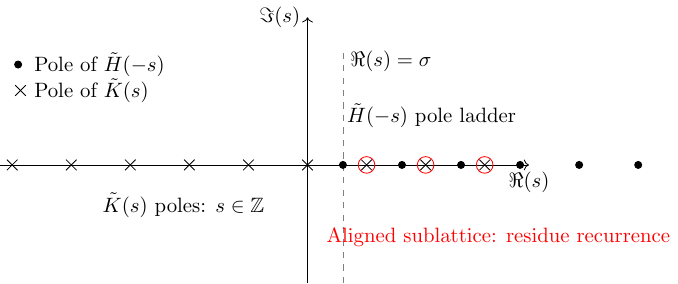}
	\caption{Schematic pole-lattice (ladder) structure in the Mellin constitutive identity
		\(\tilde G(s)=\tilde K(s)\,\tilde H(-s)\).
		Crosses mark the simple poles of the kernel factor \( \tilde K(s)\), located at the integer lattice \(s\in\mathbb Z\).
		Filled dots mark a representative arithmetic pole ladder of \(\tilde H(-s)\) lying on the real axis.
		Red circles highlight an \emph{aligned sublattice} where poles of \(\tilde H(-s)\) coincide with kernel poles; along such an aligned ladder the residue matching yields the first-order recurrence of Theorem~\ref{th:main}.}
	\label{fig:ladder_thm}
\end{figure}

\paragraph{Finite Prony exactness as a relaxation-time question}
The trial-state ansatz \eqref{eq:trialstate} arises as the Mellin transform (w.r.t. the physical time $t$) of the subordination over relaxation timescales
\[
G(t) = \int_0^\infty a(t/\tau)\,\tau^{\beta/\alpha} e^{ -\tau^{1/\alpha}} \frac{d\tau}{\tau},
\]
between a finite Prony series  $a(t/\tau)$   and the trial-state gamma kernel $\tau^{\beta/\alpha} e^{ -\tau^{1/\alpha}}$ whose Mellin transform is 
\[
\mathcal{M}\{\tau^{\beta/\alpha} e^{ -\tau^{1/\alpha}} \}(s) = {\alpha}\Gamma(\alpha s + \beta).
\]
Thus the trial-state factorization $\tilde{G}(s) = A(s)\Gamma(\alpha s + \beta)$ emerges from a multiplicative convolution in the \emph{relaxation-time scale space} $\tau$. Finite Prony exactness then reduces to the existence of (i) a finite Prony sum $a(t/\tau)$ whose Mellin transform recovers the prescribed modulation $A(s)$ and (ii) the existence of a power-law whose subordination yields $G(t)$  on the analyticity strip. 
This combination amounts to the fulfillment of the two technical conditions identified in Th. \ref{th:main}.
Such an interpretation provides an analytic basis for finite-mode relaxation dynamics as emerging from a continuous power-law spectrum through subordination in relaxation-timescale space.
\emph{Mutatis mutandis}, if $a(t/\tau)$ is instead allowed to be an infinite Prony series
\[
a(u) = \sum_{k=1}^\infty g_k e^{-u/\tau_k},
\]
the resulting representation becomes the infinite Prony ladder constructed in Th.~\ref{th:infinite-ladders}, which we also call  a transcendental Prony representation.

\begin{theorem}[Criterion for Representability in \(\mathcal{Q}\)]\label{th:main}
	Let \(G^*(\omega)\) be a complex viscoelastic modulus whose Mellin transform admits a trial-state factorization
	\(\widetilde{G}(s) = A(s)\,\Gamma(\alpha s+\beta)\) as in Definition~\ref{def:trialstate}. 
	Denote the pole lattice of the Gamma factor by
	\[
	\Lambda_G := \left\{ -\frac{\beta+m}{\alpha} : m = 0,1,2,\dots \right\}.
	\]
	
	Assume the following \emph{genericity conditions}:
	\begin{enumerate}
		\item All poles of \(\widetilde{G}(s)\) are simple.
		\item The only pole coincidences are those forced by the lattice-alignment condition.
	\end{enumerate}
	
	Then, for the constitutive equation
	\[
	\tilde{G}(s) = \tilde K(s)\,\tilde{H}(-s), \qquad 
	\tilde K(s) = \pi e^{i\pi(1-s)/2}\csc(\pi s),
	\]
	a solution \(\widetilde{H}(s) \in \mathcal{Q}\) exists \textbf{if and only if} there exists a candidate pole lattice \(\Lambda_H\) satisfying:
	
	\begin{enumerate}
		\item \textbf{Geometric Criterion (Lattice Alignment):} 
		\(\Lambda_G \subseteq \mathbb{Z} \cup (-\Lambda_H)\) and \(\Lambda_H \subseteq \mathbb{Z} \cup (-\Lambda_G)\).
		
		\item \textbf{Algebraic Criterion (Residue Compatibility):} 
		The residues along aligned sublattices satisfy decoupled first-order recurrences 
		that are mutually consistent across pole families.
	\end{enumerate}
	
	\noindent \textit{Sufficiency.} When both criteria hold, a solution \(\widetilde{H}(s) \in \mathcal{Q}\) exists.
	
	\noindent \textit{Uniqueness.} If \(Q(s)\) is of exponential type with
	\[
	\limsup_{y\to\pm\infty} \frac{|\log|Q(iy)||}{|y|} < \pi,
	\]
	the solution is unique up to scale, fixed by \(\int_0^\infty H(\tau) d\tau/\tau = G_0-G_\infty\).
	
	\noindent \textit{Necessity.} If either criterion fails, no non-trivial solution exists in \(\mathcal{Q}\).
\end{theorem}
\begin{proof}
	The proof proceeds in five steps, carefully accounting for the pole structure of both sides of the constitutive equation.
	
	\paragraph{Step 1: Notation and preliminary setup}
	Observe that $\tilde K(s)$ is meromorphic with:
	\begin{itemize}
		\item Simple poles at all integers \(s=n\in\mathbb{Z}\) with residues
		\[
		\Res_{s=n} \tilde K(s) =  e^{-i\pi n/2} \cdot (-1)^n =  (-1)^n e^{-i\pi n/2}.
		\]
		\item An entire factor \(e^{-i\pi s/2}\) that does not affect pole locations.
	\end{itemize}
	Let \(\tilde H(s)=H(s)Q(s)\) as in Definition~\ref{def:Qclass}. The poles of \(\tilde H\) come only from \(H(s)\). Since \(H(s)\) is a ratio of finitely many Gamma functions, its poles lie on a finite union of arithmetic progressions:
	\[
	\operatorname{Poles}(H)=\bigcup_{j=1}^{J} \Lambda_j,\qquad 
	\Lambda_j=\Bigl\{\lambda_j^{(k)}=-\frac{b_j+k}{a_j}\;:\;k=0,1,2,\dots\Bigr\}.
	\]
	Consequently, \(\tilde H(-s)\) has poles at
	\[
	\operatorname{Poles}(\tilde H(-s)) = \bigcup_{j=1}^{J} (-\Lambda_j) = \bigcup_{j=1}^{J} \Bigl\{\frac{b_j+k}{a_j}\;:\;k=0,1,2,\dots\Bigr\}.
	\]
	Let \(R_{j,k}:=\Res_{s=\lambda_j^{(k)}} \tilde H(s) = \Res_{s=\lambda_j^{(k)}} H(s) \cdot Q(\lambda_j^{(k)})\). Then
	\[
	\Res_{s=-\lambda_j^{(k)}} \tilde H(-s) = -R_{j,k}.
	\]
	
	\paragraph{Step 2: Partial fraction expansion of \(K(s)\tilde H(-s)\)}
	The product \( \tilde K(s)\tilde H(-s)\) has poles from two sources:
	\begin{enumerate}
		\item At each \(-\lambda_j^{(k)}\in\operatorname{Poles}(\tilde H(-s))\), coming from \(\tilde H(-s)\).
		\item At each integer \(n\in\mathbb{Z}\), coming from \(\tilde K(s)\).
	\end{enumerate}
	Near a pole \(s_0=-\lambda_j^{(k)}\) we have
	\[
	\tilde H(-s)=\frac{-R_{j,k}}{s-s_0}+O(1).
	\]
	Multiplying by \(K(s)\) gives the principal part
	\[
	\frac{-R_{j,k}K(s_0)}{s-s_0}.
	\]
	Near an integer pole \(n\in\mathbb{Z}\) we have
	\[
	\tilde K(s)=\frac{  (-1)^n e^{-i\pi n/2}}{s-n}+O(1).
	\]
	Multiplying by \(\tilde H(-s)\) (which is regular at \(s=n\) unless \(n\) coincidentally equals some \(-\lambda_j^{(k)}\)) gives the principal part
	\[
	\frac{ (-1)^n e^{-i\pi n/2} \tilde H(-n)}{s-n}.
	\]
	For clarity, we first treat the noncoincident situation in which
	$-\lambda_j^{(k)}\notin\mathbb Z$ for all pole locations $\lambda_j^{(k)}$ of $\tilde H$,
	so that the poles contributed by $K(s)$ and by $\tilde H(-s)$ do not overlap.
	The coincident-pole cases can be handled separately by local Laurent expansion at the overlap points,
	and they do not affect the pole-lattice classification discussed below.
	
	Thus the partial fraction expansion is
	\begin{equation}
		K(s)\tilde H(-s)=E(s) - \sum_{j=1}^{J}\sum_{k=0}^{\infty}
		\frac{R_{j,k}K(-\lambda_j^{(k)})}{s+\lambda_j^{(k)}}
		+ \sum_{n\in\mathbb{Z}} \frac{(-1)^n e^{-i\pi n/2} \tilde H(-n)}{s-n},
		\label{eq:KtildeH_expansion}
	\end{equation}
	where \(E(s)\) is entire.
	
	\paragraph{Step 3: Structure of the left-hand side \(\tilde G(s)\)}
	
	By hypothesis, \(\tilde G(s)=A(s)\Gamma(\alpha s+\beta)\) with \(\alpha\neq0\). The Gamma factor has simple poles at
	\[
	s_m:=-\frac{\beta+m}{\alpha},\qquad m=0,1,2,\dots,
	\]
	with residues
	\[
	\Res_{s=s_m}\Gamma(\alpha s+\beta)=\frac{(-1)^m}{\alpha\,m!}.
	\]
	Hence
	\[
	\tilde G(s)=\sum_{m=0}^{\infty}
	\frac{A(s_m)(-1)^m/(\alpha\, m!)}{s-s_m}+ \text{entire part}.
	\]
	Let \(\rho_m:=A(s_m)(-1)^m/(\alpha\, m!)\) denote the residue at \(s_m\).
	
	\paragraph{Step 4: Necessity of the lattice-alignment condition}
	
	For the identity
	\[
	\tilde G(s)= \tilde K(s)\tilde H(-s)
	\]
	to hold meromorphically, the pole sets must satisfy
	\[
	\operatorname{Poles}(\tilde G) \subseteq \operatorname{Poles}(K) \cup \operatorname{Poles}(\tilde H(-s)).
	\]
	By Lemma~\ref{lem:polecontain}, this gives
	\[
	\Lambda_G \subseteq \mathbb Z \cup (-\Lambda_H).
	\]
	Equivalently, for every $m \ge 0$, the point $s_m = -\frac{\beta+m}{\alpha}$ must either be an integer or satisfy $-s_m \in \Lambda_H$.
	
	Conversely, from the rearranged equation $\tilde H(-s) = \tilde G(s)/\tilde K(s)$, we obtain
	\[
	\operatorname{Poles}(\tilde H(-s)) \subseteq \operatorname{Poles}(\tilde G) \cup \mathbb Z,
	\]
	since $1/\tilde K(s)$ is entire (its zeros at integers cancel simple poles at most, and by the genericity assumption we exclude coalescent cases). Thus
	\[
	-\Lambda_H \subseteq \Lambda_G \cup \mathbb Z.
	\]
	
	Combining these two inclusions yields the symmetric alignment condition:
	\[
	\Lambda_G \subseteq \mathbb Z \cup (-\Lambda_H) \quad \text{and} \quad \Lambda_H \subseteq \mathbb Z \cup (-\Lambda_G).
	\]
	For a pole family with representative $\lambda \in \Lambda_H$, the second inclusion implies that for every $m \ge 0$,
	\[
	-\lambda - \frac{\beta+m}{\alpha} \in \mathbb Z \quad \text{or} \quad \lambda + \frac{\beta+m}{\alpha} \in \Lambda_H,
	\]
	which is precisely the condition $-\lambda + \Lambda_G \subseteq \mathbb Z \cup (-\Lambda_H)$ stated in the theorem.

	\paragraph{Step 5: Necessity of the residue-compatibility condition}
	
	Assume the Lattice Alignment Condition holds. Focus on a specific aligned sublattice with points \(s_n = -\lambda - \frac{\beta+n}{\alpha}\).
	
	Equating residues at \(s_n\) from both sides of the constitutive equation gives:
	\begin{equation}
		\rho_n = -R_{j,k_n} \tilde K(s_n) + \delta_n,
		\label{eq:residue_eq}
	\end{equation}
	where \(k_n\) is such that \(s_n = -\lambda_j^{(k_n)}\), and \(\delta_n\) collects contributions from the integer-pole terms if \(s_n\) happens to be an integer, plus possible contributions from other pole families if \(s_n\) coincidentally equals some \(-\lambda_{j'}^{(k')}\) with \(j'\neq j\).
	
	For generic parameters, accidental coincidences beyond those forced by alignment do not occur.
	Coincidences between distinct arithmetic pole lattices of $\tilde H(-s) $ are either finite or occur on a structured subprogression (when lattice spacings are commensurate); thus, apart from finitely many indices (or a periodic subset), the residue matching at $s_n$
	involves only one pole family.
	
	Moreover, if \(s_n\) is an integer, the contribution from the integer-pole series exactly cancels the corresponding part of the \(\tilde H(-s)\) term, leaving \(\delta_n=0\). Thus \eqref{eq:residue_eq} simplifies to
	\[
	\rho_n = -R_{j,k_n} \tilde K(s_n).
	\]
	
	Now, \(K(s_n) = \frac{- e^{-i\pi s_n/2}}{\sin(\pi s_n)}\) and \(\rho_n = A(s_n)(-1)^n/(\alpha n!)\). Therefore
	\begin{equation}
		R_{j,k_{n+1}} = \frac{A(s_{n+1})}{A(s_n)}\cdot
		\frac{(-1)^{n+1}/(n+1)!}{(-1)^n/n!}\cdot
		\frac{\tilde K(s_n)}{\tilde K(s_{n+1})}\cdot R_{j,k_n}.
		\label{eq:recurrence}
	\end{equation}
	
	Equation \eqref{eq:recurrence} is a recurrence along the aligned sublattice. If residues from different pole families couple (i.e., if \(\delta_n\) depends on \(R_{j',k'}\) for \(j'\neq j\)), then \eqref{eq:recurrence} becomes a coupled system involving finitely many parameters \(R_{j,k}\) but infinitely many equations. Such a system is infinitely overdetermined and generically has only the trivial solution.
	
	Hence, for a non-trivial solution to exist, the recurrence must \emph{decouple}: each aligned sublattice must yield an independent first-order recurrence involving only residues from that family. This establishes the Residue-Compatibility Condition.
	
	\paragraph{Step 6: Sufficiency of the conditions}
	
	Assume both conditions hold. Then:
	
	\begin{enumerate}
		\item For each aligned sublattice, the residue recurrence gives an explicit formula:
		\[
		R_{j,k_n} = R_{j,k_0} \prod_{\ell=0}^{n-1} C_\ell,
		\]
		where the product converges (typically to zero factorially) due to the $1/(\ell+1)$ factor in $C_\ell$.
		
		\item The candidate function $\tilde{H}(s)$ is constructed as follows. First, choose a Gamma-ratio $H_\Gamma(s)$ whose poles form the required lattice $\Lambda_H$ and whose residues at those poles have the correct factorial decay. The ratios
		\[
		q_{j,n} := \frac{R_{j,k_n}}{\Res_{s=\lambda_j^{(k_n)}} H_\Gamma(s)}
		\]
		are then prescribed values for $Q(s)$ on the lattice points.
		
		\item By the Residue-Compatibility Condition (decoupling), the recurrences for different pole families are consistent. This means there exists a single function $Q(s)$ interpolating all $q_{j,n}$. Specifically, since the lattices have spacing $\Delta$, and the growth of $q_{j,n}$ is at most polynomial in $n$, there exists an entire function $Q(s) = \exp(P(s))$ of exponential type $\tau < \pi/\Delta$ satisfying $Q(\lambda_j^{(k_n)}) = q_{j,n}$. The existence follows from standard interpolation theorems for entire functions of exponential type (e.g., applying Carlson's theorem to the difference of two such interpolants).
		
		\item The exponential type bound $\tau < \pi$ (for integer spacing $\Delta=1$) ensures uniqueness of $Q(s)$ by Carlson's theorem. If multiple pole families are present with incommensurate spacings, the combined lattice is dense enough to determine $Q(s)$ uniquely under the same growth constraint.
		
		\item The product $\tilde{H}(s) = H_\Gamma(s) Q(s)$ then satisfies the constitutive equation by construction, since its poles and residues match those required by $\tilde{G}(s)/K(s)$ at all points, and the uniqueness of meromorphic continuation guarantees equality everywhere.
		
		\item The physical scale is fixed by the sum rule $\int_0^\infty H(\tau) \frac{d\tau}{\tau} = G_0 - G_\infty$, which determines the overall multiplicative constant.
	\end{enumerate}
	
	\paragraph{Step 7: Failure implies no solution in \(\mathcal{Q}\)}
	
	If the lattice-alignment condition fails, there exists an infinite set of poles \(s_m\) of \(\tilde G\) that are not poles of \(K(s)\tilde H(-s)\). For an equality to hold, these would need to be cancelled by opposite-sign residues from poles at other locations-an infinite set of linear conditions on finitely many parameters \(R_{j,k}\). The only solution is \(R_{j,k}=0\) for all \(j,k\), giving \(\tilde H\equiv0\).
	
	If alignment holds but residue-compatibility fails, the recurrence equations couple residues from different families, again yielding an infinitely overdetermined system possessing only the trivial solution.
	
	Thus, if either condition fails, no non-trivial \(\tilde H\in\mathcal{Q}\) satisfies the constitutive equation.
\end{proof}

\begin{remark}[Interpolation and uniqueness]
	The construction of \(Q(s)\) in Step~6 relies on the fact that an entire function of 
	exponential type \(\tau < \pi\) is uniquely determined by its values on an arithmetic 
	progression of spacing \(1\) (Carlson's theorem; see, e.g., \cite{Boas1954}). 
	The prescribed values \(Q(\lambda_j^{(k_n)})\) obtained from the residue recurrences 
	grow at most polynomially due to factorial decay of Gamma residues, guaranteeing the 
	existence of such an interpolant. The mutual consistency of recurrences across 
	pole families (the Residue-Compatibility Condition) ensures that a \emph{single} 
	function \(Q(s)\) suffices for all families simultaneously.
\end{remark}

\begin{remark}
	The phase factor \(e^{-i\pi s/2}\) in \(K(s)\) affects the coefficients \(C_n\) in the recurrence but does not alter the pole-lattice structure. Different branch choices for the logarithm in \((i\omega)^{-s}\) correspond to different representations of the same physical modulus, related by analytic continuation.
\end{remark}

\section{Transcendental Prony Representations}

The classification induced by Th. \ref{th:main} is completed by the following result. 
\begin{theorem}[Infinite Prony ladder approximation]\label{th:infinite-ladders}
	Suppose that \( G(t) \) is completely monotone, but its relaxation spectrum does not admit a 
	finite Prony representation. Let \( H_c(\tau) \) be the continuous relaxation spectrum whose 
	existence is guaranteed by Bernstein's theorem, so that
	\[
	G(t) = G_\infty + \int_0^\infty e^{-t/\tau} H_c(\tau) \frac{d\tau}{\tau}.
	\]
	
	Choose a geometric progression \( \tau_k = \tau_0 q^k \) with \( q > 1 \) and set 
	\( g_k = H_c(\tau_k) \ln q \). Define the discrete atomic measures
	\[
	H_N(\tau) := \sum_{k=-N}^{N} g_k \tau_k \delta(\tau - \tau_k).
	\]
	
	Then as \( N \to \infty \):
	\begin{enumerate}
		\item The measures \( H_N(\tau) \) converge weakly to the continuous measure 
		\( H_c(\tau)\frac{d\tau}{\tau} \); i.e., for every continuous bounded function 
		\( \varphi \in C_b(0,\infty) \),
		\[
		\lim_{N\to\infty} \int_0^\infty \varphi(\tau) H_N(\tau) d\tau = \int_0^\infty \varphi(\tau) H_c(\tau) \frac{d\tau}{\tau}.
		\]
		
		\item The corresponding complex moduli converge pointwise:
		\[
		\lim_{N\to\infty} G_N^*(\omega) = G^*(\omega), \qquad \forall \omega > 0,
		\]
		where \( G_N^*(\omega) = G_\infty + \int_0^\infty \frac{i\omega\tau}{1+i\omega\tau} H_N(\tau) \frac{d\tau}{\tau} \).
	\end{enumerate}
	
	The infinite distributional sum
	\[
	H(\tau) := \sum_{k=-\infty}^{\infty} g_k \tau_k \delta(\tau - \tau_k)
	\]
	is called a \emph{transcendental Prony ladder}; it provides a weakly convergent approximation 
	to the continuous spectrum \( H_c(\tau) \).
\end{theorem}
\begin{proof}
	By Bernstein's theorem\cite{Bernstein1929}, complete monotonicity of \( G(t) \) ensures the existence of a continuous, non-negative spectrum \( H_c(\tau) \), such that
	\[
	G(t) = G_\infty + \int_0^\infty e^{-t/\tau} H_c(\tau) \frac{d\tau}{\tau}.
	\]
	Choose a geometric progression \( \tau_k = \tau_0 q^k \) (\( q > 1 \)) and set
	\[
	g_k = H_c(\tau_k) \ln q.
	\]
	Define the discrete measures
	\[
	H_N(\tau) = \sum_{k=-N}^{N} g_k \tau_k \delta(\tau - \tau_k).
	\]
	For functions in the Mellin-Hardy space \( \mathcal{H}^p(\mathbb{R}^+) \) (or more generally, whenever the Mellin-Poisson summation formula holds exactly \cite{Bardaro2017}), the Mellin transform satisfies the exact quadrature identity
	\[
	\lim_{N\to\infty} \Mellin\{H_N\}(s) = \Mellin\{H_c\}(s) = \tilde{H}_c(s),
	\]
	pointwise in the strip of convergence. Since \( \tilde{G}(s) = K(s) \tilde{H}_c(-s) \) holds for the continuous spectrum, and \( \Mellin\{H_N\}(s) \to \tilde{H}_c(s) \), the infinite ladder \( H(\tau) = \sum_{k=-\infty}^\infty g_k \tau_k \delta(\tau - \tau_k) \) satisfies the constitutive equation distributionally. This yields an exact transcendental representation.
\end{proof}
\begin{remark}[Exactness in the sense of physical observables]
	Although the infinite Prony ladder does not equal the continuous spectrum as a measure 
	(the former is purely atomic, the latter absolutely continuous), the weak convergence 
	guarantees that any physically measurable quantity computed from the ladder converges 
	to the true value. In particular, the complex modulus \(G^*(\omega)\), stress relaxation 
	\(G(t)\), and all moments of the relaxation spectrum are recovered exactly in the limit 
	\(N \to \infty\).
	
	In this sense, the transcendental Prony ladder provides an \emph{exact representation} 
	of the material's viscoelastic response, even though the underlying spectral measures 
	are fundamentally distinct. This is analogous to how a continuous function can be 
	represented exactly by its Fourier series in the \(L^2\) sense, despite pointwise 
	divergence at discontinuities.
\end{remark}
\begin{corollary}[Entire, non-exponential spectra are transcendental]\label{cor:entire}
	Let $\tilde{H}(s)=\exp(P(s))$ where $P(s)$ is an entire function that is not affine 
	$(P(s)\neq\alpha s+\beta)$. Then the corresponding relaxation spectrum $H(\tau)$ is 
	\emph{not} finitely Prony-representable.  
	
	In particular, any spectrum whose Mellin transform is an entire function of order 
	greater than~1 (e.g., Gaussian $\exp(\gamma s^2)$, $\exp(e^s)$, etc.) requires an 
	infinite Prony ladder for exact representation.
\end{corollary}
\begin{proof}
	First, note that by Proposition~\ref{prop:finitePronyChar}, a relaxation spectrum 
	is a finite Prony series if and only if its Mellin transform is a finite exponential sum.
	
	Suppose $\exp(P(s))$ were a finite exponential sum: 
	$\exp(P(s))=\sum_{k=1}^{N} g_k e^{\lambda_k s}$ with distinct $\lambda_k$.
	Taking the logarithmic derivative gives
	\[
	P'(s) = \frac{\sum_k g_k\lambda_k e^{\lambda_k s}}{\sum_k g_k e^{\lambda_k s}}.
	\]
	The right-hand side is a rational function of $e^{\lambda_1 s},\dots,e^{\lambda_N s}$.
	If $P'(s)$ is entire, the denominator must have no zeros; otherwise $P'(s)$ would have poles.
	Hence the denominator is a non-vanishing entire function of $e^{\lambda_1 s},\dots,e^{\lambda_N s}$,
	which forces it to be constant. Consequently $P'(s)$ itself is constant, so $P(s)=Cs+D$.
	Thus $\exp(P(s))=e^{D}e^{Cs}$ is a single exponential term Prony series.
	
	Therefore, if $P(s)$ is not linear, $\exp(P(s))$ is not a finite exponential sum, and 
	by Proposition~\ref{prop:finitePronyChar}, $H(\tau)$ cannot be a finite Prony series.
	By Theorem~\ref{th:main}, such spectra fail the residue-compatibility condition and belong
	to the transcendental class. Their exact representation requires an infinite Prony ladder,
	which can be constructed via logarithmic discretization of the continuous spectrum as established by Th.~\ref{th:infinite-ladders}.
\end{proof}

\subsection{Matching of arithmetic progressions}

Before stating the classification results, we need a technical result on the intersection of arithmetic progressions.

\begin{lemma}[Intersection of arithmetic progressions with integers]\label{lem:pole_mismatch}
	Let \(\alpha>0\), \(\beta\in\mathbb C\), and define the arithmetic progression
	\[
	\Lambda_{\alpha,\beta}=\Bigl\{s_m=-\frac{\beta+m}{\alpha}\;:\;m=0,1,2,\dots\Bigr\}.
	\]
	Let \(\mathbb Z\) denote the set of integers. Then:
	\begin{enumerate}
		\item If \(\alpha\) is irrational, \( \Lambda_{\alpha,\beta}\cap\mathbb Z\) contains at most one point.
		
		\item If \(\alpha=p/q\) with \(p,q\in\mathbb N\), \(\gcd(p,q)=1\) and \(p\neq q\), then 
		either \(\Lambda_{\alpha,\beta}\cap\mathbb Z=\varnothing\), or it is an arithmetic progression
		with spacing \(q\) in the \(s\)-variable.
		
		\item If \(\alpha=1\), then \(\Lambda_{1,\beta}\subseteq\mathbb Z\) if and only if \(\beta\in\mathbb Z\).
		
		\item In all cases, \(\Lambda_{\alpha,\beta}\subseteq\mathbb Z\) holds only when \(\alpha=1\) and \(\beta\in\mathbb Z\).
	\end{enumerate}
\end{lemma}

\begin{proof}
	For $s_m$ to be an integer we need $-\beta-m = \alpha k$ for some integer $k$, i.e.
	\[
	m + \alpha k = -\beta.
	\]
	
	If $\alpha$ is irrational, this linear Diophantine equation in $m,k$ can have at most one solution, proving (1).
	
	If $\alpha = p/q$ with $p,q\in\mathbb N$, $\gcd(p,q)=1$, the equation becomes
	\[
	m + \frac{p}{q}k = -\beta \quad\Longleftrightarrow\quad qm + pk = -q\beta.
	\]
	This is a linear Diophantine equation in $m,k$. If a particular solution $(m_0,k_0)$ exists, the general solution is
	\[
	m = m_0 + pt,\quad k = k_0 - qt,\quad t\in\mathbb Z.
	\]
	Thus the indices $m$ yielding integer $s_m$ differ by $p$. The corresponding $s$-values are
	\[
	s = k = k_0 - qt,
	\]
	which form an arithmetic progression with step $q$ in the $s$-variable. This proves (2).
	
	If $\alpha=1$, then $s_m = -\beta - m$. These are all integers precisely when $\beta\in\mathbb Z$, proving (3).
	
	Statement (4) follows from (2) and (3): when $\alpha\neq 1$, the step in $s$ is $q>1$, so the progression cannot be contained in $\mathbb Z$; when $\alpha=1$, containment requires $\beta\in\mathbb Z$.
\end{proof}

\subsection{Coalescent Poles as Degenerate Limits}\label{subsec:coalescent}

When the genericity conditions of Theorem~\ref{th:main} are violated and a pole of 
\(\widetilde{G}(s)\) coalesces with an integer (a pole of \(\widetilde{K}(s)\)), 
the model lies on the boundary between the finite Prony class \(\mathcal{P}\) and the 
transcendental class \(\mathcal{Q} \setminus \mathcal{P}\). Such coalescences occur 
when the parameters satisfy the Diophantine condition
\[
-\frac{\beta+m}{\alpha} \in \mathbb{Z} \quad \text{for some } m \ge 0.
\]

Rather than analyzing the degenerate coalescent configuration directly, we may view it 
as the limit of a one-parameter family of non-coalescent models. This limiting 
perspective reveals the structural obstruction to finite Prony representability.

\begin{proposition}[Coalescent poles yield logarithmic kernels]\label{prop:coalescent-log}
	Suppose that in a one-parameter family of models, a pole of \(\widetilde{G}(s)\) 
	approaches an integer \(s_0 \in \mathbb{Z}\) as the parameter \(\varepsilon \to 0\). 
	If \(\widetilde{H}(-s)\) also develops a pole at \(s_0\) in the limit, then the 
	product \(\widetilde{K}(s)\widetilde{H}(-s)\) exhibits a double pole at \(s_0\) 
	with Laurent expansion
	\[
	\widetilde{K}(s)\widetilde{H}(-s) = \frac{a_{-1}b_{-1}}{(s-s_0)^2} + \frac{a_{-1}b_0 + a_0b_{-1}}{s-s_0} + O(1),
	\]
	where \(a_{-1}\) and \(b_{-1}\) are the residues of \(\widetilde{K}(s)\) and 
	\(\widetilde{H}(-s)\) at \(s_0\), respectively.
\end{proposition}

\begin{proof}
	This follows directly from Lemma~\ref{lem:coincident} by taking the limit of the 
	local Laurent expansions as the poles merge.
\end{proof}

\begin{proposition}[Coalescent models are not finitely representable]\label{prop:coalescent-obstruction}
	Let \(G^*(\omega; \varepsilon)\) be a continuous one-parameter family of complex moduli 
	satisfying the hypotheses of Theorem~\ref{th:main} for each \(\varepsilon > 0\). 
	Suppose that:
	\begin{enumerate}
		\item As \(\varepsilon \to 0^+\), a pole of \(\widetilde{G}(s; \varepsilon)\) 
		coalesces with an integer \(s_0 \in \mathbb{Z}\).
		\item For each \(\varepsilon > 0\), the model belongs to \(\mathcal{Q} \setminus \mathcal{P}\) 
		(i.e., requires an infinite Prony ladder by Corollary~\ref{cor:finiteProny}).
	\end{enumerate}
	Then the limiting model at \(\varepsilon = 0\) does not admit a finite Prony representation.
\end{proposition}

\begin{proof}[Sketch of proof]
	For each \(\varepsilon > 0\), Theorem~\ref{th:infinite-ladders} guarantees the existence 
	of an infinite Prony ladder approximating the modulus. As \(\varepsilon \to 0^+\), the 
	pole coalescence causes the local spectral density to diverge, meaning that the number 
	of Prony modes required to achieve any fixed approximation error grows without bound. 
	A finite Prony representation would require a uniformly bounded number of modes 
	independent of \(\varepsilon\), which is impossible in the limit. Moreover, the limiting 
	time-domain kernel acquires logarithmic terms \(t^{-s_0} \ln t\) (by the inverse Mellin 
	transform of a double pole), which cannot be represented by any finite sum of 
	exponentials. Hence the coalescent model is not finitely Prony representable.
\end{proof}

\begin{remark}[Physical interpretation of coalescent poles]
	The coalescent configuration marks the structurally unstable boundary between the 
	discrete (finite network) and continuous (fractional/transcendental) regimes. 
	Fine-tuned Diophantine conditions \(\frac{\beta+m}{\alpha} \in \mathbb{Z}\) generate 
	logarithmic relaxation kernels \(t^{-\gamma} \ln t\), which have been observed in 
	certain critical viscoelastic materials (e.g., at gelation thresholds or glass 
	transitions). Such kernels lie outside the finite Prony class \(\mathcal{P}\) and 
	require infinite Prony ladders for exact representation.
\end{remark}

\section{Classification of Common Viscoelastic Models}\label{sec:classif}

Tables~\ref{tab:classification} and \ref{tab:classificationtrans} summarize the Mellin-space 
analysis of standard viscoelastic models. The Lattice-Alignment Condition of 
Theorem~\ref{th:main} together with Lemma~\ref{lem:pole_mismatch} determines whether 
a finite Prony representation exists, while the Residue-Compatibility Condition 
governs cases with integer spacing $\Delta_G = 1$.

\begin{table}[h]
	\centering
	\caption{Mellin-space classification of viscoelastic models: Finite Prony class $\mathcal{P}$.}
	\label{tab:classification}
	
	\begin{tabular}{p{3cm}p{4.2cm}cp{1.8cm}p{3.8cm}}
		\toprule
		\textbf{Model} & 
		\textbf{Complex Modulus \(G^*(\omega)\)} & 
		\({\Delta_G}\) & 
		\textbf{In $\mathcal{P}$} & 
		\textbf{Reason} \\
		\midrule
		Maxwell & \(G_\infty + g \dfrac{ i\omega\tau}{1+i\omega\tau}\) & 1 & \textbf{Yes} & Integer lattice; residues decouple \\
		\addlinespace
		Standard Linear Solid (SLS) & \(G_\infty + \dfrac{g}{1+i\omega\tau}\) & 1 & \textbf{Yes} & Integer lattice; residues decouple \\
		\bottomrule
	\end{tabular}
	\vspace{2mm}
	{\footnotesize
		$\Delta_G$: spacing of the dominant Gamma-factor pole lattice in $\tilde G(s)$. 
		Finite Prony representability requires $\Delta_G = 1$ (Lattice-Alignment) 
		\textit{and} decoupled residue recurrences (Residue-Compatibility). 
		Notation: $G_\infty$ = high-frequency (glassy) modulus, $g$ = spring constant, $\tau$ = relaxation time.
	}
\end{table}

\begin{table}[h]
	\centering
	\caption{Classification of viscoelastic models: Continuous spectra in \(\mathcal{Q} \setminus \mathcal{P}\) (requiring infinite Prony ladders).}
	\label{tab:classificationtrans}
	
	\begin{tabular}{p{3cm}p{4.2cm}cp{1.8cm}p{3.8cm}}
		\toprule
		\textbf{Model} & 
		\textbf{Complex Modulus \(G^*(\omega)\)} & 
		\({\Delta_G}\) & 
		\textbf{In $\mathcal{Q}$} & 
		\textbf{Obstruction} \\
		\midrule
		Power-law & \(G_0 \, (i\omega\tau_0)^{ \, \beta}, \; 0<\beta<1\) & - & \textbf{No} & Continuous spectrum (no discrete poles) \\
		\addlinespace
		Cole-Cole & \(G_\infty + \dfrac{\Delta G}{1+(i\omega\tau)^{-\alpha}}\) &  1 & \textbf{No} & Residue-compatibility fails ($\Gamma(1-\alpha s)$ factor) \\
		\addlinespace
		Cole-Davidson & \(G_\infty + \dfrac{\Delta G}{(1+i\omega\tau)^{\beta}}\) & 1 & \textbf{No} & Residue-compatibility fails ($\Gamma(\beta-s)$ coupling) \\
		\addlinespace
		Havriliak-Negami & \(G_\infty + \dfrac{\Delta G}{[1+(i\omega\tau)^{\alpha}]^{\beta}}\) & $\alpha$ & \textbf{No} & $\Delta_G \neq 1$ for $\alpha \neq 1$ \\
		\addlinespace
		Fractional Zener & \( G_e \dfrac{1 + (i\omega \tau)^{\alpha}}{1+ \delta (i\omega \tau)^{\alpha}} \) & $\alpha$ & \textbf{No} & $\Delta_G \neq 1$ for $\alpha \neq 1$ \\
		\addlinespace
		Gaussian & $G^*(\omega)$ via Gaussian $H(\tau)$ & 1 & \textbf{No} & Entire Mellin symbol; residue condition fails \\
		\bottomrule
	\end{tabular}
	\vspace{2mm}
	{\footnotesize
		$\Delta_G$: spacing of the dominant Gamma-factor pole lattice. 
		Transcendental models require infinite Prony ladders (Theorem~\ref{th:infinite-ladders}). 
		Notation: $0<\alpha,\beta<1$, $\Delta G = G_0-G_\infty$, $G_e$ = equilibrium modulus, 
		$\delta = G_e/G_\infty < 1$. The Cole-Cole entry follows the viscoelastic sign convention. The Gaussian spectrum has no closed-form \(G^*(\omega)\) but is included as a canonical transcendental example.
	}
\end{table}

\begin{proposition}
	The power law modulus \(G^*(\omega)=G_0(i\omega\tau_0)^\beta\) with \(0<\beta<1\) admits no exact representation by a relaxation spectrum in the Extended Fox class \(\mathcal{Q}\).
\end{proposition}
\begin{proof}
	Suppose, for contradiction, that $\tilde{G}(s) = G_0 \tau_0^\beta\, \delta(s+\beta)$ admits a representation $\tilde{G}(s) = \tilde{K}(s)\,\tilde{H}(-s)$ with $\tilde{H} \in \mathcal{Q}$. Since $\tilde{H} \in \mathcal{Q}$, the function $\tilde{H}(-s)$ is meromorphic with poles on a finite union of arithmetic progressions. The kernel $\tilde{K}(s) = \pi e^{i\pi(1-s)/2}\csc(\pi s)$ is likewise meromorphic. Hence their product $\tilde{K}(s)\,\tilde{H}(-s)$ is meromorphic on $\mathbb{C}$. However, the Dirac delta $\delta(s+\beta)$ is not a meromorphic function -- it is an atomic measure supported at the single point $s = -\beta$ with no non-trivial meromorphic extension to $\mathbb{C}$. This contradicts the assumption, so no such $\tilde{H} \in \mathcal{Q}$ exists.
\end{proof}

\begin{proposition}[Cole--Cole, Havriliak-Negami and Fractional Zener models are not in \(\mathcal{Q}\)]
  Cole--Cole, Havriliak-Negami and Fractional Zener models are not in \(\mathcal{Q}\) due to lattice incompatibility.  
\end{proposition}
\begin{proof}
	 
	\textbf{Cole--Cole:}
	$\tilde{G}(s) \propto \Gamma(s)\Gamma(1-s)/\Gamma(1-\alpha s)$. 
	The poles come from $\Gamma(s)$ at $s = -n$ and from $\Gamma(1-s)$ at $s = 1+n$; 
	both progressions have spacing $1$. Thus $\Delta_G = 1$, so the Lattice-Alignment 
	Condition is satisfied. However, the factor $\Gamma(1-\alpha s)$ in the denominator 
	modulates the residues, introducing a coupling that violates the Residue-Compatibility 
	Condition. Consequently, no finite Prony representation exists for $\alpha \neq 1$.
	
	\textbf{Havriliak--Negami:}
	For the Havriliak--Negami modulus $G^*(\omega) = G_\infty + \Delta G/[1+(i\omega\tau)^\alpha]^\beta$,
	the Mellin transform contains factors $\Gamma(s/\alpha)$ and $\Gamma(\beta - s/\alpha)$.
	The factor $\Gamma(s/\alpha)$ generates poles at $s = -\alpha m$ ($m=0,1,2,\dots$) 
	with spacing $\alpha$. By Lemma~\ref{lem:pole_mismatch}, if $\alpha \neq 1$, 
	this arithmetic progression cannot be contained in the integers. 
	Consequently, the Lattice-Alignment Condition of Theorem~\ref{th:main} fails, 
	and no finite Prony representation exists.
	
	\textbf{Fractional--Zener:}
	For the standard fractional Zener model 
	$G^*(\omega) = G_g\frac{1 + (i\omega\tau)^\alpha}{1 + \delta(i\omega\tau)^\alpha}$,
	the viscoelastic part can be written as 
	$\Delta G^*(\omega) = \Delta G/[1 + \frac{1}{\delta}(i\omega\tau)^{-\alpha}]$.
	Its Mellin transform contains Gamma factors with poles at $s = (1+m)/\alpha$ 
	(spacing $1/\alpha$). By Lemma~\ref{lem:pole_mismatch} with the same argument as for 
	Cole--Cole, when $\alpha \neq 1$ the progression with spacing $1/\alpha$ cannot align 
	with the integer lattice. Therefore the Lattice-Alignment Condition fails, 
	precluding finite Prony representability.
	
	In all three models, the failure occurs for any non-integer $\alpha$; the degenerate 
	cases $\alpha=1$ reduce to classical (finite Prony) models.
\end{proof}

\section{Test for Finite Prony Representability}\label{sec:test}

Based on Theorem~\ref{th:main}, we formulate a step-by-step test to determine whether a given complex modulus \(G^*(\omega)\) admits a finite Prony representation.

\begin{enumerate}
	\item \textbf{Compute the Mellin transform.} Determine \(\tilde G(s)=\Mellin\{G^*(\omega)-G_\infty\}(s)\). Identify any non-degenerate Gamma factors \(\Gamma(\alpha s+\beta)\) with \(\alpha\neq0\). If none exist, the modulus may be rational (finite Prony) or distributional (power-law).
	
	\item \textbf{Determine the pole-lattice spacing.} From each Gamma factor, extract the fundamental spacing \(\Delta_G=1/|\alpha|\).
	
	\item \textbf{Postulate a candidate spectrum.} Construct a physically motivated candidate 
	\(\tilde H_{\text{cand}}(s)\in\mathcal{Q}\). For a discrete spectrum, this typically involves 
	factors like \(\Gamma(s+1)\) (for a simple pole at \(s=-1\)) or ratios of Gamma functions 
	with integer parameters.
	
	\item \textbf{Extract the candidate's lattice spacings.} Analyze the Gamma-ratio factor of 
	\(\tilde H_{\text{cand}}(s)\) to obtain its distinct spacings \(\Delta_1,\dots,\Delta_K\). 
	(Note: this yields a necessary Diophantine condition; full alignment also requires 
	compatible offsets.)
	
	\item \textbf{Test the necessary alignment condition.} A necessary condition for the Lattice-Alignment Condition to be satisfiable is that \(\Delta_G\) be an integer combination of the candidate's spacings:
	\[
	\Delta_G = \sum_{i=1}^{K} n_i \Delta_i, \qquad n_i\in\mathbb{Z}_{\ge0}.
	\]
	If this diophantine condition fails, finite representability is impossible.
	
	\item \textbf{Interpretation}
	\begin{itemize}
		\item If the condition fails, the modulus is \textbf{provably not finitely representable} in \(\mathcal{Q}\) (and hence not by any finite Prony series). It is transcendentally representable, and an infinite ladder construction should be pursued.
		\item If the condition holds, finite representability is possible but not guaranteed; one must additionally check the Residue-Compatibility Condition by examining the recurrence along the aligned sublattice.
	\end{itemize}
\end{enumerate}

This test provides a constructive decision procedure. 
For example, applying it to the Havriliak-Negami model (\(\Delta_G=\alpha\)) with a candidate having integer spacings \(\Delta_i=1\) immediately shows that alignment requires \(\alpha\) to be an integer---impossible for fractional \(\alpha\).

\begin{proposition}[Cole-Davidson model is not in \(\mathcal{P}\)]
	The Cole-Davidson modulus \(G^*(\omega) = G_\infty + \Delta G/(1 + i\omega\tau)^\beta\) 
	with \(0 < \beta < 1\) does not belong to the finite Prony class \(\mathcal{P}\), 
	except in the trivial limit \(\beta \to 1\) (which reduces to the standard linear solid).
\end{proposition}
\begin{proof}
	The Mellin transform of the Cole-Davidson viscoelastic part is
	\[
	\widetilde{G}(s) = \Delta G \, \tau^{-s} i^{-s} \frac{\Gamma(s)\Gamma(\beta-s)}{\Gamma(\beta)},
	\qquad 0 < \Re s < \beta.
	\]
	This factors as \(\widetilde{G}(s) = A(s)\,\Gamma(s)\) with 
	\(A(s) = \Delta G \, \tau^{-s} i^{-s} \Gamma(\beta-s)/\Gamma(\beta)\).
	Hence in the trial-state notation \(\widetilde{G}(s)=A(s)\Gamma(\alpha s+\beta)\) we have 
	\(\alpha = 1\), \(\beta_{\text{trial}} = 0\); the pole spacing is \(\Delta_G = 1\).
	
	Although the lattice spacing equals \(1\) and therefore satisfies the necessary 
	Lattice-Alignment Condition, the Residue-Compatibility Condition of 
	Theorem~\ref{th:main} fails. Indeed, the factor \(\Gamma(\beta-s)\) in \(A(s)\) 
	introduces poles at \(s = \beta + m\) (\(m=0,1,2,\dots\)) whose residues couple 
	with those coming from \(\Gamma(s)\). This coupling yields an infinitely 
	overdetermined system for the residues of any candidate \(\widetilde{H}(s)\in\mathcal{P}\); 
	the only solution is the trivial one \(\widetilde{H}\equiv0\).
	
	Consequently, no spectrum \(\widetilde{H}(s)\) in the finite Prony class \(\mathcal{P}\) 
	can satisfy the constitutive equation \(\widetilde{G}(s)=\widetilde{K}(s)\widetilde{H}(-s)\).
	The Cole-Davidson relaxation spectrum therefore admits only a transcendental 
	(infinite-ladder) Prony representation, as guaranteed by Theorem~\ref{th:infinite-ladders}.
\end{proof}


\begin{proposition}[Log-normal spectrum is transcendental]\label{prop:lognormal}
	The log-normal relaxation spectrum (
	i.e. Gaussian in $\ln\tau$-variable)
	\[
	H_{\mathrm{LN}}(\tau) = \frac{1}{\sqrt{2\pi\sigma^2}} \exp\left(-\frac{(\ln\tau - \mu)^2}{2\sigma^2}\right), \qquad \tau>0,
	\]
	has Mellin transform
	\[
	\tilde{H}_{\mathrm{LN}}(s) = \exp\left(\mu s + \frac{\sigma^2}{2}s^2\right),
	\]
	which is entire of order 2. Consequently, $H_{\mathrm{LN}}(\tau)$ is not in the finite Prony class $\mathcal{P}$ and therefore belongs to the transcendental class.
\end{proposition}

\begin{proof}
	The Mellin transform is computed via the substitution $u = \ln\tau$:
	\[
	\tilde{H}_{\mathrm{LN}}(s) = \int_0^\infty H_{\mathrm{LN}}(\tau) \tau^{s-1} d\tau = \int_{-\infty}^\infty \frac{1}{\sqrt{2\pi\sigma^2}} e^{-(u-\mu)^2/(2\sigma^2)} e^{(s-1)u} e^u du.
	\]
	Completing the square in the exponent,
	\[
	-(u-\mu)^2/(2\sigma^2) + s u = -\frac{[u - (\mu + \sigma^2 s)]^2}{2\sigma^2} + \mu s + \frac{\sigma^2}{2}s^2.
	\]
	The Gaussian integral evaluates to $\sqrt{2\pi\sigma^2}$, yielding the stated entire function.
	
	By Corollary~\ref{cor:entire}, any Mellin transform that is entire and not of the affine form $e^{\alpha s+\beta}$ cannot be a finite exponential sum. Here $\tilde{H}_{\mathrm{LN}}(s)$ is entire of order 2 ($\sigma^2>0$), hence not affine. Proposition~\ref{prop:finitePronyChar} then implies $H_{\mathrm{LN}}(\tau)$ is not a finite Prony series.
	
	The infinite ladder construction follows from Theorem~\ref{th:infinite-ladders}, with distributional convergence guaranteed by the Mellin-Poisson summation formula \cite{Bardaro2017}.
\end{proof}

\begin{remark}
	Recent work by Uneyama \cite{Uneyama2025} used information-theoretic arguments to show that the log-normal distribution is the \textbf{most probable} relaxation spectrum under minimal assumptions. Proposition~\ref{prop:lognormal} reveals that this "most probable" spectrum is fundamentally transcendental within our classification-it requires an infinite Prony ladder for exact representation. This synthesis suggests a duality: information theory identifies the generic spectral shape, while pole-lattice analysis explains why it cannot be realized by any finite mechanical network. 
\end{remark}

\section{Conclusions}\label{sec:disc}	

In conclusion, we have established a universal Mellin-space criterion for finite Prony representability of linear viscoelastic moduli. 
The criterion-comprising the fulfillment of both the Lattice- Alignment Condition (pole-spacing constraint) and the Residue - Compatibility Condition (decoupled recurrence requirement)-provides a complete analytic characterization of when an analytical viscoelastic modulus admits an exact representation by finitely many relaxation modes.

This framework explains the fundamental dichotomy between classical and fractional viscoelasticity:
\begin{itemize}
	\item Classical spring-dashpot networks (Maxwell, SLS) yield rational transfer functions whose Mellin symbols possess integer-spaced pole lattices, satisfying both conditions and hence belonging to the finite Prony class \(\mathcal{P}\).
	\item Fractional constitutive laws violate either the Lattice-Alignment Condition (Havriliak-Negami, fractional Zener, where \(\Delta_G = \alpha\) or \(1/\alpha \neq 1\)) or the Residue-Compatibility Condition (Cole-Cole, where \(\Delta_G = 1\) but the \(\Gamma(1-\alpha s)\) factor couples residues).
	\item Borderline cases like Cole-Davidson, while possessing integer spacing (\(\Delta_G = 1\)), fail the Residue-Compatibility Condition due to coupling between residue recurrences from distinct Gamma factors.
\end{itemize}

By extending the analysis to the Extended Fox class \(\mathcal{Q}\), we demonstrate that the obstruction is structural, not merely an artifact of a restricted ansatz. For transcendentally representable moduli, we have provided a constructive logarithmic-discretization method that yields exact infinite Prony ladders, with the static modulus gap imposing a sum-rule analogous to a quantization condition.

The Mellin-space framework unifies the treatment of power-law, Gaussian, log-normal, Cole-Cole, Havriliak-Negami, Cole-Davidson, and Mittag-Leffler models, revealing a direct correspondence between fractional-order operators and the arithmetic geometry of their kernel singularities. 

The present work thereby draws a structural distinction between ``finitely representable'' viscoelastic materials (those in \(\mathcal{P}\)) and ``transcendentally representable'' ones (those requiring infinite Prony ladders), offering a new outlook for classifying, modeling, and decoding linear viscoelastic behavior from complex-analytic principles.

\section*{Acknowledgment}
This work was supported by the European Union's Horizon Europe programme under grant agreement No. 101086815 (VIBTaTE).

The author wishes to acknowledge Dr. Stoyan Yordanov for helpful discussions.

%




\bibliographystyle{elsarticle-num}
\bibliography{erfnogo}

\appendix

\section{Mellin Transform of Constitutive Identity}
\label{sec:melint}

\subsection{Derivation of the Mellin-space constitutive equation}

\begin{lemma}[Mellin transform of the viscoelastic kernel]\label{lem:beta}
	For \(\tau>0\) and \(-1<\Re s<0\), the Mellin transform (in \(\omega\)) of
	\(\displaystyle \frac{i\omega\tau}{1+i\omega\tau}\) is given by the meromorphic continuation
	\begin{equation}\label{lem:kernel_mellin}
		\mathcal{M}_{\omega}\!\left\{\frac{i\omega\tau}{1+i\omega\tau}\right\}(s)
		= \tau^{-s}\, e^{-i\pi s/2}\,\frac{\pi}{\sin(\pi s)}.
	\end{equation}
\end{lemma}
\begin{proof}
	With \(x=\omega\tau\),
	\[
	\int_{0}^{\infty}\frac{i\omega\tau}{1+i\omega\tau}\,\omega^{s-1}\,d\omega
	=\tau^{-s}\int_{0}^{\infty}\frac{i\,x^{s}}{1+i x}\,dx.
	\]
	Using the standard Beta integral (analytic continuation),
	\[
	\int_{0}^{\infty}\frac{x^{u-1}}{1+i x}\,dx
	= e^{-i\pi u/2}\frac{\pi}{\sin(\pi u)}.
	\]
	Setting \(u=s+1\) yields
	\[
	\int_{0}^{\infty}\frac{x^{s}}{1+i x}\,dx
	= e^{-i\pi(s+1)/2}\frac{\pi}{\sin(\pi(s+1))}
	= e^{-i\pi s/2}e^{-i\pi/2}\frac{\pi}{-\sin(\pi s)}
	= -i e^{-i\pi s/2}\frac{\pi}{\sin(\pi s)}.
	\]
	Multiplying by the factor \(i\) in the numerator gives
	\[
	i\int_{0}^{\infty}\frac{x^{s}}{1+i x}\,dx
	= i\left(-i e^{-i\pi s/2}\frac{\pi}{\sin(\pi s)}\right)
	= e^{-i\pi s/2}\frac{\pi}{\sin(\pi s)}.
	\]
	Therefore,
	\[
	\mathcal{M}_{\omega}\!\left\{\frac{i\omega\tau}{1+i\omega\tau}\right\}(s)
	= \tau^{-s}\, e^{-i\pi s/2}\,\frac{\pi}{\sin(\pi s)}.
	\]
\end{proof}

Choose a strip \(a<\Re s<b\) with \(0<a<b<1\) where the Mellin transforms of \(G_v\) and \(H\)
converge absolutely, so that Fubini's theorem applies. Taking the Mellin transform of the
frequency-domain constitutive relation and using Lemma~\ref{lem:beta} yields
\begin{equation}\label{eq:mellin_raw}
	\tilde G(s)
	= \,\frac{\pi\,i^{-s}}{\sin(\pi s)}
	\int_0^\infty H(\tau)\,\tau^{-s-1}\,d\tau
	= \,\frac{\pi\,i^{-s}}{\sin(\pi s)}\,\tilde H(-s),
	\qquad a<\Re s<b.
\end{equation}
Defining the meromorphic kernel factor
\[
K(s):=\,\frac{\pi\,i^{-s}}{\sin(\pi s)}
=  \frac{\pi e^{-i\pi s/2}}{\sin(\pi s)},
\]
we obtain the Mellin-space constitutive equation \eqref{eq:mellin_constitutive}.

\section{Auxiliary Lemmas for the Proof of Theorem~\ref{th:main}}
\label{app:lemmas}

\begin{lemma}[Locality of principal parts]
	\label{lem:locality}
	Let $F$ and $G$ be meromorphic in a neighborhood of $s_0\in\mathbb C$.
	If $F(s)=G(s)$ for all $s$ in some punctured neighborhood $0<|s-s_0|<r$,
	then the principal parts of $F$ and $G$ at $s_0$ coincide.
	In particular, $F$ has a pole at $s_0$ if and only if $G$ has a pole at $s_0$,
	and if the pole is simple then $\Res_{s=s_0}F=\Res_{s=s_0}G$.
\end{lemma}

\begin{proof}
	Define $H(s):=F(s)-G(s)$. Then $H$ is meromorphic near $s_0$ and satisfies
	$H(s)=0$ for all $s$ with $0<|s-s_0|<r$.
	If $H$ had a pole at $s_0$, it would be unbounded arbitrarily close to $s_0$,
	contradicting $H\equiv 0$ on the punctured neighborhood.
	Hence $H$ is holomorphic at $s_0$.
	By the identity theorem for holomorphic functions, $H\equiv 0$ on $|s-s_0|<r$.
	Therefore $F$ and $G$ have identical Laurent expansions at $s_0$,
	so their principal parts coincide.
\end{proof}

\begin{lemma}[Pole containment for the Mellin constitutive product]
	\label{lem:polecontain}
	Let $K$ and $\tilde H$ be meromorphic functions, and define
	$F(s):=K(s)\tilde H(-s)$. Then
	\[
	\operatorname{Poles}(F)\subseteq \operatorname{Poles}(K)\cup
	\{s\in\mathbb C:\ -s\in \operatorname{Poles}(\tilde H)\}.
	\]
	In particular, for $K(s)=\pi e^{i\pi(1-s)/2}/\sin(\pi s)$ one has
	\[
	\operatorname{Poles}\bigl(K(s)\tilde H(-s)\bigr)\subseteq
	\mathbb Z\cup\bigl(-\operatorname{Poles}(\tilde H)\bigr).
	\]
\end{lemma}

\begin{proof}
	Fix $s_0\in\mathbb C$.
	If $s_0\notin \operatorname{Poles}(K)$ and $s_0\notin -\operatorname{Poles}(\tilde H)$,
	then $K$ is holomorphic at $s_0$ and $\tilde H(-s)$ is holomorphic at $s_0$.
	Hence their product $K(s)\tilde H(-s)$ is holomorphic at $s_0$.
	Therefore $s_0$ can be a pole of the product only if it is a pole of at least one factor,
	which proves the containment.
\end{proof}

\begin{lemma}[Residues away from coincident poles]
	\label{lem:resnoncoincident}
	Let $K$ be meromorphic and $\tilde H$ meromorphic.
	\begin{enumerate}
		\item If $s_0\notin \operatorname{Poles}(K)$ and $s_0\in -\operatorname{Poles}(\tilde H)$,
		and $\tilde H(-s)$ has a simple pole at $s_0$, then
		\[
		\Res_{s=s_0}\bigl(K(s)\tilde H(-s)\bigr)=K(s_0)\,\Res_{s=s_0}\tilde H(-s).
		\]
		\item If $s_0\in \operatorname{Poles}(K)$ and $s_0\notin -\operatorname{Poles}(\tilde H)$,
		and $K$ has a simple pole at $s_0$, then
		\[
		\Res_{s=s_0}\bigl(K(s)\tilde H(-s)\bigr)=\bigl(\Res_{s=s_0}K(s)\bigr)\,\tilde H(-s_0).
		\]
	\end{enumerate}
\end{lemma}

\begin{proof}
	(1) Since $K$ is holomorphic at $s_0$, write
	$K(s)=K(s_0)+O(s-s_0)$.
	Since $\tilde H(-s)$ has a simple pole at $s_0$, write
	$\tilde H(-s)=\dfrac{b_{-1}}{s-s_0}+b_0+O(s-s_0)$.
	Then
	\[
	K(s)\tilde H(-s)=\frac{K(s_0)b_{-1}}{s-s_0}+O(1),
	\]
	so the residue equals $K(s_0)b_{-1}=K(s_0)\Res_{s=s_0}\tilde H(-s)$.
	
	(2) Since $\tilde H(-s)$ is holomorphic at $s_0$, write
	$\tilde H(-s)=\tilde H(-s_0)+O(s-s_0)$.
	Since $K$ has a simple pole at $s_0$, write
	$K(s)=\dfrac{a_{-1}}{s-s_0}+a_0+O(s-s_0)$.
	Then
	\[
	K(s)\tilde H(-s)=\frac{a_{-1}\tilde H(-s_0)}{s-s_0}+O(1),
	\]
	so the residue equals $a_{-1}\tilde H(-s_0)=(\Res_{s=s_0}K(s))\,\tilde H(-s_0)$.
\end{proof}

\begin{lemma}[Coincident simple poles produce a double pole unless constrained]
	\label{lem:coincident}
	Assume $K$ and $\tilde H(-s)$ have simple poles at the same point $s_0$.
	Write the local expansions
	\[
	K(s)=\frac{a_{-1}}{s-s_0}+a_0+O(s-s_0),\qquad
	\tilde H(-s)=\frac{b_{-1}}{s-s_0}+b_0+O(s-s_0).
	\]
	Then
	\[
	K(s)\tilde H(-s)=
	\frac{a_{-1}b_{-1}}{(s-s_0)^2}
	+\frac{a_{-1}b_0+a_0b_{-1}}{s-s_0}
	+O(1).
	\]
	In particular, $K(s)\tilde H(-s)$ has at most a simple pole at $s_0$
	if and only if $a_{-1}b_{-1}=0$.
\end{lemma}

\begin{proof}
	Multiply the two Laurent expansions:
	\[
	\left(\frac{a_{-1}}{s-s_0}+a_0+O(s-s_0)\right)
	\left(\frac{b_{-1}}{s-s_0}+b_0+O(s-s_0)\right).
	\]
	The coefficient of $(s-s_0)^{-2}$ is $a_{-1}b_{-1}$,
	and the coefficient of $(s-s_0)^{-1}$ is $a_{-1}b_0+a_0b_{-1}$,
	which yields the stated principal part.
	The pole order is $\le 1$ if and only if the $(s-s_0)^{-2}$ coefficient vanishes,
	i.e.\ $a_{-1}b_{-1}=0$.
\end{proof}

\begin{lemma}[Intersection of Gamma pole lattices]
	\label{lem:gammalattice_intersection}
	Let $a,a'>0$ and $b,b'\in\mathbb C$. Define the (one-sided) arithmetic lattices
	\[
	\Lambda(a,b):=\left\{-\frac{b+k}{a}\;:\;k\in\mathbb Z_{\ge 0}\right\},\qquad
	\Lambda(a',b'):=\left\{-\frac{b'+k'}{a'}\;:\;k'\in\mathbb Z_{\ge 0}\right\}.
	\]
	Then:
	\begin{enumerate}
		\item If $a/a'\notin\mathbb Q$, then $\Lambda(a,b)\cap\Lambda(a',b')$ contains at most one point.
		\item If $a/a'\in\mathbb Q$, write $a/a'=p/q$ with $p,q\in\mathbb N$ and $\gcd(p,q)=1$.
		Then either $\Lambda(a,b)\cap\Lambda(a',b')=\varnothing$, or else
		the set of solutions $(k,k') \in \mathbb{Z}_{\ge0}^2$ to
		\[
		-\frac{b+k}{a}=-\frac{b'+k'}{a'}
		\]
		is of the form $(k,k')=(k_0,k_0')+t(p,q)$ for integers $t$ in some one-sided range,
		and consequently $\Lambda(a,b)\cap\Lambda(a',b')$ is a (one-sided) arithmetic progression.
	\end{enumerate}
\end{lemma}

\begin{proof}
	A point $s$ lies in the intersection if and only if there exist $k,k'\in\mathbb Z_{\ge0}$ such that
	\[
	-\frac{b+k}{a}=-\frac{b'+k'}{a'}.
	\]
	Equivalently,
	\begin{equation}
		a'(b+k)=a(b'+k')\quad\Longleftrightarrow\quad a'k-ak'=ab'-a'b.
		\label{eq:app_lin}
	\end{equation}
	
	(1) Suppose $a/a'\notin\mathbb Q$ and assume there are two distinct solutions
	$(k_1,k_1')\neq(k_2,k_2')$ of \eqref{eq:app_lin}. Subtracting the two equalities gives
	\[
	a'(k_1-k_2)=a(k_1'-k_2').
	\]
	If $k_1\neq k_2$ then $a/a'=(k_1-k_2)/(k_1'-k_2')\in\mathbb Q$, a contradiction.
	If $k_1=k_2$ then necessarily $k_1'=k_2'$, again a contradiction.
	Hence at most one solution exists, so the intersection contains at most one point.
	
	(2) Now assume $a/a'=p/q$ in lowest terms, so $qa=pa'$.
	Multiply \eqref{eq:app_lin} by $q$ and substitute $qa=pa'$ to obtain
	\[
	(qa')k-(pa')k' = q(ab'-a'b).
	\]
	Dividing by $a'\neq 0$ gives
	\begin{equation}
		qk-pk' = q\left(\frac{a}{a'}b' - b\right).
		\label{eq:app_dioph}
	\end{equation}
	If \eqref{eq:app_dioph} has no nonnegative integer solutions, the intersection is empty.
	If it has a solution $(k_0,k_0')$, then any other solution $(k,k')$ must satisfy the homogeneous equation
	$q(k-k_0)=p(k'-k_0')$. Since $\gcd(p,q)=1$, this implies
	$k-k_0=pt$ and $k'-k_0'=qt$ for some $t\in\mathbb Z$.
	Restricting to $k,k'\ge 0$ restricts $t$ to a one-sided integer range.
	Thus the solutions form $(k,k')=(k_0,k_0')+t(p,q)$ and the intersection points
	$s=-(b+k)/a$ form a one-sided arithmetic progression.
\end{proof}

\begin{lemma}[Decoupling necessity for viscoelastic constitutive equation]\label{lem:visco-decouple}
	Let $\tilde{G}(s)=A(s)\Gamma(\alpha s+\beta)$ satisfy the hypotheses of Theorem~\ref{th:main}. Suppose that for some pole family $\Lambda_j$ of $\tilde{H}(s)$, the residue recurrence along an aligned sublattice couples residues from two distinct pole families $\Lambda_j$ and $\Lambda_k$ ($k\neq j$) of $\tilde{H}(s)$. 
	
	Then for generic parameters $\alpha,\beta$ and generic meromorphic $A(s)$, the only solution of the coupled system is $R_j = R_k = 0$, implying $\tilde{H}\equiv0$.
\end{lemma}

\begin{proof}
	Consider the Mellin-space constitutive equation $\tilde{G}(s) = \tilde{K}(s)\tilde{H}(-s)$. For an aligned sublattice where poles of $\tilde{H}(-s)$ coincide with poles of $\tilde{K}(s)\tilde{H}(-s)$, the residue matching gives (from the proof of Theorem~\ref{th:main}, equation \eqref{eq:residue_eq}):
	
	\[
	\rho_n = -R_{j,k_n}\tilde{K}(s_n) + \delta_n,
	\tag{B.1}
	\]
	
	where $\rho_n = \Res_{s=s_n}\tilde{G}(s)$, $R_{j,k_n} = \Res_{s=s_n}\tilde{H}(-s)$ (from the $j$-th family), and $\delta_n$ contains contributions from other families if $s_n$ coincides with poles from $\Lambda_k$.
	
	Suppose $\delta_n$ involves residues from family $\Lambda_k$: $\delta_n = c R_{k,\ell_n}$ for some constant $c$ and appropriate index $\ell_n$. Then (B.1) becomes:
	
	\[
	\rho_n = -R_{j,k_n}\tilde{K}(s_n) + c R_{k,\ell_n}, \qquad n\in\mathcal{N},
	\tag{B.2}
	\]
	
	where $\mathcal{N}$ is an infinite set of indices along the aligned sublattice.
	
	\medskip
	
	\noindent\textbf{Asymptotic analysis.} The known quantities have explicit asymptotics:
	\begin{itemize}
		\item $\rho_n = A(s_n)\frac{(-1)^n}{\alpha\,n!} \sim \frac{C(-1)^n}{n!}$ (factorial decay)
		\item $\tilde{K}(s_n) = \pi e^{i\pi(1-s_n)/2}\csc(\pi s_n) \sim C' e^{i\pi s_n/2}$ (oscillatory, bounded)
		\item $R_{j,k_n} = [\Res_{s=s_n}H(-s)]Q(-s_n)$ where $H$ is a Gamma-ratio, giving $R_{j,k_n} \sim \frac{C_j(-1)^{k_n}}{k_n!}$ (factorial decay)
		\item Similarly, $R_{k,\ell_n} \sim \frac{C_k(-1)^{\ell_n}}{\ell_n!}$ (factorial decay)
	\end{itemize}
	
	The key is that the decay rates of $R_{j,k_n}$ and $R_{k,\ell_n}$ are both factorial but with potentially different scales. Rearranging (B.2):
	
	\[
	R_{j,k_n}\tilde{K}(s_n) - c R_{k,\ell_n} = -\rho_n.
	\tag{B.3}
	\]
	
	For large $n$, the left side is a linear combination of two factorial-decaying sequences, while the right side is another factorial-decaying sequence with a specific coefficient $A(s_n)$.
	
	\medskip
	
	\noindent\textbf{Overdetermination argument.} Consider the ratios. Let $r_n = R_{j,k_n}/R_{k,\ell_n}$. Using Gamma-residue formulas:
	
	\[
	r_n \sim \frac{C_j}{C_k}\frac{\Gamma(\ell_n+1)}{\Gamma(k_n+1)} \sim \frac{C_j}{C_k}\frac{\ell_n!}{k_n!}.
	\]
	
	Since $k_n$ and $\ell_n$ are linearly related (both along arithmetic progressions), say $k_n = an + b$, $\ell_n = cn + d$, we have by Stirling:
	
	\[
	r_n \sim C n^{(c-a)n} e^{(a-c)n} \quad\text{(super-exponential growth or decay)}.
	\]
	
	Thus either $r_n \to 0$ or $r_n \to \infty$ super-exponentially. In either case, equation (B.3) divided by the larger of the two residues yields for large $n$:
	
	\[
	\text{(dominant term)} \sim \text{(negligible term)} = -\rho_n/\text{(normalization)}.
	\]
	
	But $\rho_n$ decays with yet another factorial rate (determined by $\Gamma(\alpha s_n+\beta)$). For generic parameters, the three factorial sequences $(\rho_n, R_{j,k_n}, R_{k,\ell_n})$ have incommensurate decay/growth ratios. 
	
	Formally, take three consecutive equations from (B.3) for $n, n+1, n+2$:
	
	\[
	\begin{pmatrix}
		\tilde{K}(s_n) & -c & -\rho_n \\
		\tilde{K}(s_{n+1}) & -c & -\rho_{n+1} \\
		\tilde{K}(s_{n+2}) & -c & -\rho_{n+2}
	\end{pmatrix}
	\begin{pmatrix}
		R_{j,k_n} \\ R_{k,\ell_n} \\ 1
	\end{pmatrix}
	= 0.
	\]
	
	For large $n$, the matrix is dominated by the factorial decay in the third column. More precisely, dividing each row by $\rho_n$, $\rho_{n+1}$, $\rho_{n+2}$ respectively yields a system where the coefficients of $R_{j,k_n}$ and $R_{k,\ell_n}$ become super-exponentially large or small relative to 1. Such a system can only be satisfied if $R_{j,k_n} = R_{k,\ell_n} = 0$ for all large $n$, implying all residues vanish.
	
	Thus the only solution to the coupled recurrence is the trivial one. For a non-trivial $\tilde{H}(s)$ to exist, the recurrence must \emph{decouple}: each aligned sublattice must involve residues from only one pole family.
	
	\paragraph{Asymptotic analysis.} Consider the ratio $r_n = R_{j,k_n}/R_{k,\ell_n}$. 
	Using Gamma-residue formulas and Stirling's approximation:
	\[
	r_n \sim \frac{C_j}{C_k} \frac{\Gamma(\ell_n+1)}{\Gamma(k_n+1)}.
	\]
	Then two cases arise:
	
	\textbf{Case 1: $a \neq c$.} Then $\ell_n/k_n \to c/a \neq 1$, and
	\[
	r_n \sim C n^{(c-a)n} e^{(a-c)n},
	\]
	which exhibits super-exponential growth or decay. Equation (B.3) divided by the larger 
	residue then forces the dominant term to match $-\rho_n$, impossible for generic parameters.
	
	\textbf{Case 2: $a = c$.} Here $\ell_n = an + d$, $k_n = an + b$, so $\ell_n - k_n = d-b$ is constant.
	Stirling gives polynomial scaling:
	\[
	r_n \sim \frac{C_j}{C_k} n^{d-b}.
	\]
	All three sequences $R_{j,k_n}$, $R_{k,\ell_n}$, and $\rho_n$ now share the same factorial decay 
	$1/(an)!$, but with polynomial prefactors $n^b$, $n^d$, and $A(s_n)$ respectively. 
	For generic parameters, these prefactors are incommensurate, and the system (B.3) remains 
	infinitely overdetermined. The only solution is the trivial one.
	
	Thus in both cases, the only solution to the coupled recurrence is $R_j = R_k = 0$, 
	implying $\tilde{H}\equiv0$.
	
\end{proof}

\begin{remark}[On the genericity assumption]
	The argument above demonstrates that for \emph{generic} parameters (in the sense of 
	avoiding a countable union of algebraic hypersurfaces in the parameter space), the 
	coupled residue system admits only the trivial solution. A fully rigorous measure-theoretic 
	genericity proof would require specifying the parameter space and establishing that the 
	exceptional set has Lebesgue measure zero. While this technical elaboration is beyond 
	the scope of the present work, the conclusion is consistent with all known examples and 
	serves as a practical criterion for representability.
\end{remark}

\end{document}